\documentclass[12pt]{article}
\usepackage{amsmath}
\usepackage{amssymb}
\usepackage{graphicx}
\begin{document}
\bibliographystyle{plain}
\setlength{\baselineskip}{1.5\baselineskip}
\newtheorem{theorem}{Theorem}
\newtheorem{lemma}{Lemma}
\newtheorem{corollary}{Corollary}
\newcommand{\boldm}[1]{\mbox{\boldmath$#1$}}
\newcommand{\bigO}[1]{{\mathcal{O}}\left( {#1 }\right)}
\newcommand{\bigOp}[1]{{\mathcal{O}}_p\left( {#1} \right)}
\newcommand{\smallO}[1]{{o}\left( {#1} \right)}
\newcommand{\sign}[1]{{sign}\left( {#1} \right)}
\newcommand{\smallOp}[1]{{o}_p\left( {#1} \right)}
\newcommand{\thmref}[1]{Theorem~\ref{#1}}
\newcommand{\corref}[1]{Corollary~\ref{#1}}
\renewcommand{\baselinestretch}{1}
\newenvironment{proof}[1][Proof]{\noindent\textbf{#1.} }{\ \rule{0.5em}{0.5em}}

\title{Some Theorems on Optimality of a Single Observation Confidence Interval for
the Mean of a Normal Distribution} 
\author{Stephen Portnoy{$ ^1 $}  \\ \\
Dedicated to the memory of Charles Stein (1920 - 2016) \\ }
\date{February, 2017 (updated: May, 2018)}

\maketitle 

\bigskip

\begin{abstract}

We consider the problem of finding a proper confidence interval for the mean based on a 
single observation from a normal distribution with both mean and variance unknown.  
Portnoy (2018) characterizes the scale-sign invariant rules and shows that the Hunt-Stein construction 
provides a randomized invariant rule that improves on any given randomized rule  in the sense that it 
has greater minimal coverage among 
all procedures with a fixed expected length. Mathematical results here provide a specific mixture
of two non-randomized invariant rules that achieve the minimax optimality. A multivariate
confidence set based on a single observation vector is also developed.

\end{abstract}

\footnotetext[1]{\noindent 
Professor, Department of Statistics, University
of Illinois at Urbana-Champaign \\
\smallskip
$\quad$ corresponding email: sportnoy@illinois.edu \\ }

\medskip

\newpage

\section{Introduction and basic result}

Consider a single observation $\, X \, \sim \, {\cal{N}}(\mu , \, \sigma^2)$.  Let $\, \lambda = \mu/\sigma \,$ and
note that $\, X / \sigma \,  \sim \, {\cal{N}}(\lambda, \, 1) \, $.
 
Now consider the following confidence intervals: let $\, c_1 < c_2 \,$ and define the interval
\begin{equation} \label{CI*}
CI^* \equiv CI^*(X \, ; \,\, c_1 , \, c_2 ) \, = \, \left\{
 \begin{array}{cc}
   c_1 X  \, \leq \, \mu \, \leq \, c_2  X  & \quad   X > 0  \\
   c_2 X  \, \leq \, \mu \, \leq \, c_1  X  & \quad   X < 0
  \end{array}
  \right.
 \end{equation}
 
 Portnoy (2018) provided the following coverage formula:

\begin{theorem} \label{basic}
The probability of coverage for the interval, $CI^*$ for $\lambda > 0$ is:
\begin{equation} \label{Plampos}
P(\lambda ;  \, c_1 , \, c_2)  \, = \, \left\{
 \begin{array}{cc}
   \Phi \left( \lambda \left(1 - \frac{1}{c_2} \right) \right) \, + \, 1 \, - \, \Phi \left( \lambda \left(1 + \frac{1}{c_1} \right) \right) & 
      \quad   c_1 \leq 0 \, ; \,\, c_2 \geq 0  \\
  \Phi \left( \lambda \left(1 - \frac{1}{c_2} \right) \right) \, - \, \Phi \left( \lambda \left(1 + \frac{1}{c_1} \right) \right)  & 
      \quad  c_1 > 0 \, ; \,\, c_2 > 0 \,\, . 
  \end{array}
  \right.
 \end{equation}
 Note that the first line above holds for $c_1 = 0 $ and/or $c_2 = 0$ by taking limits
 as $\, c_1 \nearrow 0 \,$ and/or $\, c_2 \searrow 0 \,$.
The coverage probability for other cases is given from these results by symmetry.
\end{theorem}

\smallskip

Portnoy (2018) also characterizes the scale-sign invariant rules as having the form of $CI^*$ and 
provides a version of the Hunt-Stein Theorem (Hunt and Stein, 1945, also see Lehmann, 1959)
to show that for any (randomized) confidence interval,
there is a randomized invariant rule whose minimal coverage (over the parameters) is larger and whose
expected length is the same.
Section 2 below finds a specific mixture of two non-randomized invariant rules that achieves minimaxity (in the
sense that it maximizes minimal coverage among all rules with fixed expectted length). Section 3
provides a brief discussion of numeric computation. Section 4 proves that there is a norm-bounded confidence set
that provides a proper confidence set for the mean based on a single (multivariate) observation from
a multivariate normal distribution with arbitrary mean and covariance matrix.

\bigskip

\section{Optimal Mixture}

The first rather complicated theorem shows that for any randomized invariant procedure there is 
a mixture of no more than 8 specific non-randomized invariant confidence intervals that is as good
(in the minimax sense above). A corollary uses linear programming theory to show
that a mixture of two specific intervals suffices. It also shows that there is a best such rule, and clearly 
this rule must be minimax (since no other rule can be strictly better).

\smallskip

\begin{theorem} 
Let $F$ be a probability distribution on $\{ c_1 < c_2 \}$ generating a randomized
invariant confidence interval. Then there are constants: $\, c_1^* \leq a_1^* \leq 0 < 1 \leq c_2^* \,$ 
and a finite mixture, $F^*$, on the intervals: $[c_1^* , \, 1]$, $[a_1^* , \, 1]$, $[c_1^* , \, c_2^*]$,
$[a_1^* , \, c_2^*]$,  $[0 , \, 1]$, $[0, \, c_2^*]$, $[1, \, c_2^*]$, and $\, \phi \,$ (the empty interval) with 
 at least as large minimal coverage probability and no larger expected length. That is,
\begin{eqnarray*} 
\inf_\lambda \, E_{F^*} \, P(\lambda ; \, C_1 , \, C_2) & \geq & 
\inf_\lambda \, E_{F} \, P(\lambda ; \, C_1 , \, C_2) \\
E_{F^*} (C_2 - C_1)  & \leq & E_{F} (C_2 - C_1) \,\, ,
\end{eqnarray*}
where $P$ denotes the coverage probability given by \eqref{Plampos}
and is repeated below for convenience.
\end{theorem}

\bigskip

\begin{proof} The proof is given by a series of lemmas.
To simplify notation, refer to the interval $CI^*(X ; \, c_1 , \, c_2)$ as $[c_1, \, c_2]$. 
By scale and sign invariance,  we can restrict to the case
$\, \lambda \geq 0 \,$ without loss of generality.

\begin{lemma} \label{cs}
The distribution $F$ can be restricted to one putting probability 1 on the set
$\, \{ [c_1 , \, c_2] : \,\, - c_2 \leq c_1 \leq 1 \, $ and $\,  c_2 > 0 \, \} \,$.
\end{lemma}

\begin{proof}
To show that we can take $\, c_2 \geq  0 \,$, set $\, c_1 = c_2 - h \,$, use Theorem 1, and consider
$$
\frac{\partial}{\partial c_2 } \, P(\lambda ;  \, c_2  - h , \, c_2 )  = 
\frac{\lambda}{c_2^2 } \, \varphi \left(\lambda \left( 1 - \frac{1}{ c_2  } \right) \right) \, - \,
\frac{\lambda}{(c_2 - h)^2 } \, \varphi \left(\lambda \left( 1 - \frac{1}{ c_2 - h } \right) \right) \, .
$$
For $c_2 < 0$, both factors of the first summand above are greater than the corresponding
factors of the second summand, and so it follows that the function
$ \, P(\lambda ;  \, c_2 - h , \, c_2)  > 0 \,$ is increasing in $c_2$.
Therefore, the interval $\, [ - (c_2 - c_1) , \, 0 ] $ has the same length but larger probability than $[c_1, \, c_2]$.
So we can take $\, c_2 \geq 0 \,$. To show the inequality is strict, we have (from Theorem 1)
\begin{eqnarray*}
P(\lambda , \, c_1 , \, 0) \,\, = \,\, 1 - \Phi \left(\lambda \left( 1 - \frac{1}{ c_1  } \right) \right) & = & 
\Phi \left(\lambda \left( \frac{1}{ c_1} - 1 \right) \right) \\
& < & \Phi \left(\lambda \left( 1 + \frac{1}{ c_1  } \right) \right)
= P(\lambda , \, 0 , \, -c_1 ) \, ,
\end{eqnarray*}
and so $c_2$ can be taken to be strictly positive.

\bigskip

A similar proof shows that we can take $c_1 \leq 1\,$: let $c_2 = c_1 + h \,$ and consider
$$
\frac{\partial}{\partial c_1 } \, P(\lambda ;  \, c_1 , \, c_1 + h)  = 
\frac{\lambda}{(c_1 + h)^2 } \, \varphi \left(\lambda \left( 1 - \frac{1}{ c_1 + h } \right) \right) \, - \,
\frac{\lambda}{c_1^2 } \, \varphi \left(\lambda \left( 1 - \frac{1}{ c_1 } \right) \right) \, .
$$
For $c_1 \geq 1$, both factors of the first summand above are smaller than the corresponding
factors of the second summand, and so it follows that the function 
$ P(\lambda ;  \, c_1 , \, c_1 + h)  < 0 \,$ is decreasing;
and so the interval $[1 , \, c_2 - c_1 +1] $ has the same length and larger probability than $[c_1, \, c_2]$.

\bigskip

Finally, to show that we can take $\, c_1 > - c_2 \,$, first note that if $c_1 \geq 0 \,$, 
the inequality is immediate (since $\, c_2 > 0 $). Next, to show that if this inequality fails, the interval
$[- c_2 , \, - c_1]$ has larger probability (and the same length) as $[ c_1 , \, c_2 ]$, define
$\, \Delta \equiv P(\lambda ;  \, c_1 , \, c_2 ) - P(\lambda ;  \, - c_2 , \, c_1) $. Let $\, b_1 = 1/c_1 \,$,
$\, b_2 = 1/c_2 \,$, and define $h$ so that $\, b_2 = - b_1 - h \,$. Note that $\, b_2 < - b_1 \,$ 
(or equivalently, $\, c_2 > - c_1 \,$ if and only if $\, h > 0 \,$. Then,
 $$
 \Delta =    \Phi \left( \lambda (1 + b_1 + h) \right) \, - \, \Phi \left( \lambda (1 - b_1) \right) 
 - \Phi \left( \lambda (1 + b_1 ) \right) \, + \, \Phi \left( \lambda (1 + -b_1 + h ) \right) 
 $$
 and
 $$
 \frac{\partial}{\partial h} \Delta = \lambda \, \varphi \left( \lambda (1 + b_1 + h) \right) \, - \,
 \varphi \left( \lambda (1 + -b_1 + h ) \right) \, > \, 0 \, .
 $$
 for $\, h > 0 \,$. Now $\, \Delta = 0 \,$ when $\, h = 0 \,$; and hence $\, \Delta \geq 0 \,$ as
 long as $\, h \geq 0 \,$. Therefore, the interval $[ c_1 , \, c_2 ]$ has larger coverage probability
 that $[ - c_2 , \, - c_1 ]$ as long as $\, h \geq 0 \,$, or equivalently $\, c_2 > - c_1 \,$.
 \end{proof}

\medskip

The following Lemma presents some derivative calculations and subsequent convexity and concavity
properties that will facilitate analyzing the coverage probabilities.
 
\begin{lemma} \label{derivs}
\begin{eqnarray} 
\frac{\partial^2 \, \Phi(\lambda ( 1 - 1/c))}{\partial c^2} & = & \frac{\lambda}{c^3} \, 
\varphi \left( \lambda \left( 1 - \frac{1}{c} \right) \right) \,
\left[ \frac{\lambda^2}{c^2} - \frac{\lambda^2}{c} - 2 \right]   \label{dP2dc2} \\
\frac{\partial \Phi(\lambda \, d)}{\partial d} & = & d \, \varphi(d \, \lambda)   \,\, ,  \label{dPdlam} \\
\frac{\partial^2 \Phi(\lambda \, d)}{\partial d^2} & = & - \lambda \,  d^3 \, \varphi(d \, \lambda)  \label{dP2dlam2}
\end{eqnarray}

From \eqref{dP2dc2}, there are functions $\, a_1(\lambda) < 0 \,$ and $\, a_2(\lambda) > 1 \,$ such that
the coverage probability $\, P(\lambda ;  \, c_1 , \, c_2) \,$ is concave in $c_1$ for $\, c_1 \leq a_1(\lambda) \,$,
convex in $c_1$ for $\, a_1(\lambda) \leq c_1 \leq 0 \,$, and 
a possibly different convex function for $\, 0 \leq c_1 \leq 1 \,$; and is convex in $c_2$ for
$\, 0 \leq c_2 \leq a_2(\lambda) \,$ and convex in $c_2$ for  $\, c_2 \geq a_2(\lambda)  \,$.

From \eqref{dP2dlam2}, $ \Phi(\lambda \, d) $ is increasing and concave in  $\lambda$ for $d \geq 0$, and 
decreasing and convex in $\lambda$ for $ d \leq 0 \,$.
\end{lemma}

\begin{proof} 
The derivative calculations are straightforward, using the fact that $\varphi'(x) = -x \varphi(x)$.
Convexity and concavity in $\lambda$ is also a trivial consequence of \eqref{dP2dlam2}. 

For the behavior of the coverage probability as a functions of $c_1$ and $c_2$, note that derivatives
of $\, P(\lambda ;  \, c_1 , \, c_2) \,$ will have the form \eqref{dP2dc2} (with arguments $c_1$ or $c_2$).
Clearly \eqref{dP2dc2} vanishes if (and only if)
\begin{equation} \label{a1def}
\frac{1}{c} = \frac{1 \pm \sqrt{1 + 8 / \lambda^2)}}{2}  \, \equiv\, \{ a_1(\lambda) < 0 , \,\, a_2(\lambda) > 1 \}
\end{equation}
Thus  $\, P(\lambda ;  \, c_1 , \, c_2) \,$ has sign changes only at $a_1(\lambda) < 0 $ and $a_2(\lambda) > 1$.
The convexity and concavity claims
follow directly by examining the behavior of 
$\, P(\lambda ;  \, c_1 , \, c_2) \,$ as $c_1$ and $c_2$ tend to $\infty$, 0, 1, 0, and $- \infty$, and noting
that the derivatives are  discontinuous at $\lambda = 0$.
\end{proof} 

\begin{lemma} \label{improve}
Given any distribution, $F(c_1 , \, c_2)$, generating a randomized invariant confidence interval, there is
a constant $\, c_2^* \,$, a 
random variable, $\, C \sim  F_1$ where $F_1$ is a distribution on $(- \infty , \, 0)$, and a family of
conditional distributions, $G_c(c_1, \, c_2)$, such that conditional on $C = c$, $G_c$ is finite discrete mixture
on the intervals: $[c , \, 1]$, $[c , \, c_2^*]$, $[0 , \, 1]$, $[0, \, c_2^*]$, $[1, \, c_2^*]$,
and $\phi$ (the empty interval); 
and such that the randomized confidence interval given by $F_1$ and $G_c$ has coverage probability 
no smaller than that of $F$ and expected length no larger that of $F$ uniformly
in $\lambda$. Furthermore, the improvement is strict unless $G_c$ is such a finite mixture. Finally,
for each fixed $\lambda$ there are functions $\, a_1(\lambda)$ (see \eqref{a1def}) and 
$\, c_1^*(\lambda) < a_1(\lambda) \,$ such that $F_1$ can be replaced by a finite discrete mixture on
$\{ c_1^*(\lambda) , \, a_1(\lambda , \, 0 ) \}$ giving no smaller coverage probability and no larger expected length
at the specific value of $\lambda$. That is, ``$c$'' in the first two intervals in the list above can be replaced by
either $\, c_1^*(\lambda) \,$ or $\, a_1(\lambda) \,$ to provide a list of 8 intervals.
Again, the improvement is strict unless $F_1$ is such a mixture.
\end{lemma}

\begin{proof}
Part 1: $c_2 \geq 1\,$.  Fix the lower endpoint, $c_1$ and let $d \equiv (1 - 1/c_2)$. Then $\, d \in [0, \, 1] \,$
and from Lemma {\ref{derivs}} (see \eqref{dP2dlam2}), the second derivative of the coverage probability is
just $\, E_F [ - \lambda \,  D^3 \, \varphi(D \, x) ] \,$; and so the coverage probability is (strictly) concave. 
Also, $\, c_2 = 1/(1-d) \,$ is convex (for $\, d \geq 1 )$. Thus, from Jensen's inequality, any $F$-probability on
$\, c_2 > 1 \,$ can be replaced by a point mass at $\, c_2^* = 1/(1 - d^*) \,$ where $\, d^* = E D  \in (0 , \, 1 )\,$ 
for which both 
$$
P(  \lambda , \, c_1 , \, c_2^*) > E_F  [ P(\lambda, \,  c_1 , \, C) \, | \,\, C > 1 ]
\quad {\mbox{\text{\rm and}}} \quad c_2^* \leq E [ C \, | \,\, C > 1]
$$
uniformly in $\lambda$.

\medskip

Part 2a: $\, 0 < c_2 \leq 1 \,$ and $ \, 0 \leq c_1 \leq 1 \,$.
From Lemma {\ref{derivs}} (see \eqref{dP2dc2}, the coverage probability is convex on $\, c_1 \leq c_2 \leq 1 \,$.
So choose $\, q \,$ so that 
$$
  1-q = E_F [ P(C_2, \, ]\lambda) \, | \,\,\, c_1 \leq C_2 \leq 1 ] \,\, .
$$.
 Then the mixture taking $\, C_2 = c_1 \,$
 with probability $q$ and $\, C_2 = 1 \,$ with probability $1 - q$ generates the empty interval,
 $\phi$, (with probability $q$) and
 the interval $[c_1,\, 1]$ (with probability $1 - q$) satisfying
 (simultaneously and uniformly in $\lambda$) both
 \begin{eqnarray*}
q \, P( \lambda  ; \phi ) + (1 - q) \, P( \lambda ; \, c_1, \, 1)  & > &
E_F  [ P( \lambda ; \, c_1 , \, C_2) \, | \, c_1 \leq C_2 \leq 1 ] \\
 \quad q \times  0 + (1-q) \times (1 - c_1) & =  & E_F  [ C_2 | \,\, c_1 \leq C_2 \leq 1 ]
\end{eqnarray*}

\medskip

Part 2b: $\, 0 \leq c_2 \leq 1 \,$ and $ \, - \infty < c_1 < 0 \,$. Again, the coverage probability is convex, and the
probability on $\, 0 \leq C_2 \leq 1$ can be replaced by a mixture on 0 and 1 (with corresponding intervals:
$[c_1 , \, 0] $ and $[c_1, \, 1]$. This provides the first part of the Lemma.

\medskip

Part 3: Finally, to replace $F_1$ by a finite discrete mixture, note that (as above) the coverage probability
is convex on $[ a_1(\lambda) , \, 0]$ and concave on $[ - \infty , \, a_1(\lambda) ]$. Thus, probability on
$[ a_1(\lambda) , \, 0]$ can be replaced by a mixture on $  a_1(\lambda) $ and 0 having larger 
coverage probability and the same (conditional) expected length. By Jensen's inequality. and probability
on $[ - \infty , \, a_1(\lambda) ]$ can be replaced by a point mass at 
$\, c_1^* \equiv E [ C_1 \, | \, - \infty  \leq  C_1 \leq a_1(\lambda) ] \,$ with larger coverage probability
(and the same conditional expected length). Note that, since $\, c_1 > - c_2^* \,$,
if $\, a_1(\lambda) \leq - c_2^* \,$ then the last interval is empty, and no point mass at $c_1^*$ is needed.
\end{proof}

\bigskip

To complete the proof of the Theorem, let $F^*$ denote the distribution generated by $F_1$
and $G_c$ given by the first part of Lemma {\ref{improve}}, and let $F_\lambda$ denote the discrete
mixture given by the last part of the Lemma. Let $\, C \sim F_1 \,$ under $F^*$ and have 
the two-point mixture on $c_1^*(\lambda)$ and $a_1(\lambda)$ under $F_\lambda$. 
Since $F_\lambda$ improves only at a fixed $\lambda$
it remains to find a rule where the improvement is uniform in $\lambda$.

From \eqref{Plampos} the coverage probability (under $F^*$ and $F_\lambda$) is a linear combinations
of  functions $\Phi(\lambda(1 - 1/c))$ where $c$ is $C$ or 0 or 1 or $c_2^*$. For $c = 0$ or $c = 1$,
the function is constant (in $\lambda$), and so the mixture probabilities will sum to provide the
coverage probability of the form
\begin{equation} \label{covmix}
P(\lambda) = b_0 \, + \, b_1 \,  \Phi \left( \lambda \left(1 - \frac{1}{d_2^*} \right) \right) \, - \, b_2 \,
E  \Phi \left( \lambda \left(1 - \frac{1}{C} \right) \right)
\end{equation}
where the expectation is under $F^*$ or $F_\lambda$. Note that the coefficients, $b_i$, are non-negative
and are exactly the same under $F^*$ and $F_\lambda$ (since the probability that $\, c < 0 \,$ is the same
under each distribution). Note that $\, (1 - 1/c_2^*) < 1 \,$ (since $\, c_2^* > 1$) and  
$\, (1 - 1/C) > 1 \,$ (since $\, C < 0$). Thus, from  \eqref{dPdlam} 
the $\lambda$-derivative of $P$ (see \eqref{covmix}) becomes
\begin{eqnarray*}
P'(\lambda) = b_1' \,  \varphi \left( \lambda \left(1 - \frac{1}{d_2^*} \right) \right) \, - \, b_2 \,
E  \left(1 - \frac{1}{C} \right) \, \varphi \left( \lambda \left(1 - \frac{1}{C} \right) \right) \\
\,\,=  \varphi(0)  \left\{ b_1' - b_2 \, E  \left(1 - \frac{1}{C} \right) \,
\exp \left( -\frac{1}{2} \lambda^2  \left[ \left(1 - \frac{1}{C} \right)^2 - \left(1 - \frac{1}{c_2^*} \right)^2 \right] 
\right) \right \} \, .
\end{eqnarray*}
Note that the coefficient of $\lambda^2$ in the exponential function is strictly positive, and so the
exponential function is monotonically (strictly) decreasing to zero. It follows (by the monotone
convergence theorem) that $P'(\lambda)$ is the difference between a positive constant and a function
that is decreasing monotonically to zero. Thus $P(\lambda)$ is monotonically increasing and is
positive for $\lambda$ large. So $P(\lambda)$ can not be minimized as $\, \lambda \rightarrow \infty \,$.
If $\, P'(0) \geq 0 \,$, $P(\lambda)$ is increasing and thus minimized at $\lambda = 0$. Otherwise, 
$P(\lambda)$ has a unique minimum at $\, \lambda = \lambda^* \in (0 ,\, \infty )$.

Case 1: $P_{F^*}'(0) > 0 \,$ and $P_{F^*}(\lambda^*)$ is minimized at $\lambda = 0$.  

As $\, \lambda \rightarrow 0 \,$, $\, a_1(\lambda) \rightarrow 0 \,$, and so $F_\lambda$ tends to
the distribution $F_{\lambda = 0}$ that puts all its probability at the point 
$\, {\tilde{c}_1} = E_{F^*} [ C \, | \, C < 0 ] \,$. By dominated convergence (and \eqref{covmix}),
$\, P_{F_\lambda}(\lambda) \rightarrow P_{F_{\lambda=0}}(0) \,$. Furthermore, since the expected length
is the same for all $F_\lambda$, $F_{\lambda=0}$ also has the same expected length as $F^*$.
As noted above, both  $P_{F^*}(0)$ and $P_{F_{\lambda=0}}(0)$ are monotonically increasing,
and so both are minimized at $\, \lambda = 0 \,$. So the interval defined using $F_{\lambda=0}$
is at least as good as that defined using $F^*$.

Case 2: $P_{F^*}'(0) < 0$ and $P_{F^*}(\lambda^*)$ is a unique minimum.

Consider small interval around $\lambda^*$. If $F^*$ not in the  family of mixtures,
coverage for $F_{\lambda^*}$ is strictly uniformly greater by  $\delta > 0$ on the interval.

\smallskip

Since coverage is bounded above by 1, one can choose $\epsilon$ (depending only on $\delta$) small enough that
$G_\epsilon \equiv (1 - \epsilon) F^* + \epsilon F_{\lambda^*} \,$ satisfies:
$$
\inf_\lambda \, E_{G_\epsilon} \, P(\lambda ; \, C_1 , \, C_2) \, > \, 
\inf_\lambda \, E_{F^*} \, P(\lambda ; \, C_1 , \, C_2) \, + \, \epsilon \delta /2  \, .
$$
See Figure 1. So $F_0$ can not be minimax for $\, \eta < \epsilon \delta / 2 \,$
except as a mixture of above form.

\bigskip

\begin{figure}[h!] \label{plots2}
\begin{center}
\includegraphics[height=3.2in, width=4.6in ]{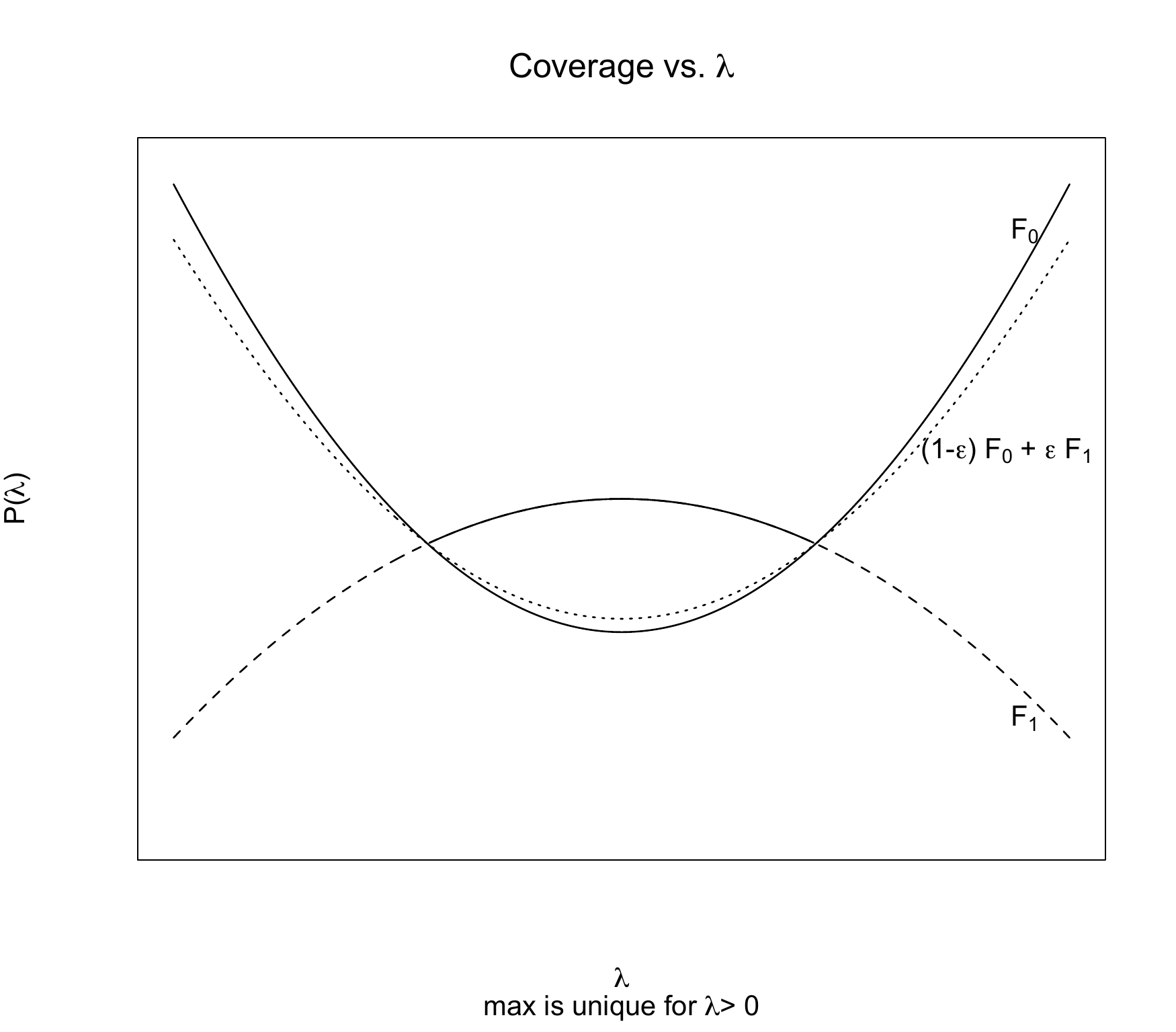}
\caption{Mixture uniformly better than assumed minimax rule.}
\end{center}
\par
\end{figure}

\bigskip\bigskip

There is one remaining issue. The mixture used in the proof above included probability mass at $\, c_2 = 0 \,$,
while the statement of the Theorem omits such mass. To complete the proof, use Lemma \ref{cs}
to replace mass on intervals of the form $[c_1 , \, p]$ (with $\, c_1 < 0 \,$) by intervals $\, [ 0 , \, -c_1 ] \,$.
Then using the transformed mass to redefine $c_2^*$ and the probabilities on the intervals $\phi$ and $\,[0 , \, c_2^* ]$,
the new distribution will provide a mixture where the only interval with its right endpoint equal to zero
is [0, 0], which is equivalent to the empty interval, $\phi$.
\end{proof}

\bigskip

\begin{corollary} Given $\, h > 0 \,$, there is a mixture of two of the 8 intervals in Lemma \ref{improve}
that is optimal in the minimax sense. From computational results described below, there
are constants $\, c_1 < a_1 \leq 0 \,$ and a probability $\, p \in [0, \, 1] \,$ such that
the $p$-mixture of $\, [c_1, \, c_2] \,$ and $\, [a_1 , \, c_2] \,$ is numerically ``minimax'',
where  $ \, c_2 \,$ is chosen so that the mixture has length $h$ (that is, $\, c_2 \,$ satisfies
$\, h = p \, (c_2 - c_1) + (1-p) \, (c_2 - a_1) \,$). Specifically, this two-point mixture numerically maximizes the 
minimal coverage probability (over $ \lambda $) among all rules with expected length $\, h \,$.
\end{corollary}

\begin{proof}
Consider any mixture of the 8 intervals given in Theorem 1, and recall from the proof of the Theorem that 
the minimum coverage over $\, \lambda \,$ occurs at a fixed value $\lambda_0$ where $\, \lambda_0 = 0 \,$ or 
is the minimizing $\lambda$-value.  Consider fixing the interval end points (say, 
$\, \{ (r_i , \, s_i ) : \,\, i = 1 , \, \cdots \, , \, 8 \, \} \,$. Then, as a function of the mixing probabilities 
$( p_1 , \, \cdots \, , \, p_8)$,  the coverage probability is
$\, \sum_{i=1}^8 p_i \, P(\lambda_0 , \, r_i , \, s_i) \,$ and the expected length is
$\, \sum_{i=1}^8 p_i \, (s_i - r_i ) \,$. Thus both the coverage probability and the expected length 
are linear in the  $\, p_i $'s. Therefore, maximizing the coverage
probability over $( p_1 , \, \cdots \, , \, p_8)$ subject to $ \, \sum_{i=1}^8 p_i \, = 1 \,$ and 
$\, \sum_{i=1}^8 p_i \, (s_i - r_i ) = h \,$ is a linear programming problem. As a consequence the coverage 
is maximized at a solution with at most two $p_i$'s non-zero. 

Thus, an optimal rule can be found by considering each pair of the 8 intervals in Theorem 2 and 
optimizing over the endpoints and mixing probability. In examining the 8-choose-2 intervals, many 
have the same form or can be obtained from others by taking limits of the endpoints or the probability. 
Also, for the length $\, h \leq 1 \,$, the intervals are of the form $[c_1 , \, 1 ]$ with $\, c_1 < 0 \,$, all of which
have coverage equal to .5. Thus, $\, c_1 = 0 \,$ minimizes the length, and by convexity  (Lemma
{\ref{improve}}), the optimal rule for $\, h \leq 1 \,$ is  a mixture of the
interval [0, 1] and the ``empty'' interval, $\phi$ (or equivalently, [0, 0]), with mixing probability $\, p = h $). 
As a consequence, only the following cases need to be treated (with the equivalent or redundant 
cases listed as ``subcases''):

\bigskip\medskip

Case 1: $\,\,  [a_1, \, 1]  \,\, [c_1,c_ 2]  \quad -1 \leq a_1 < 0 , \,\, -c2 < c1 < 0, \,\,  c2 > 1 $

$\qquad$ subcases: 
$\,\,  [c_1, \, 1] \,\, [c_1, \, c_2] , \,\,\, [c_1, \, 1] \,\,  [a_1, \, c_2] , \,\,\, [a_1 , \, 1] \,\, [a_1 , \, c_2] ,
\,\,\, [0 , \, 1 ] \,\, [ 0 , \, c_2 ]  $

 $\qquad\qquad\qquad\quad  
 [a_1, \, 1] \,\, [0, \, c_2]  ,  \,\,\, [c_1, \, 1] \,\, [0, \, c_2], \,\,\, [0, \, 1] \,\,  [c_1, \, c_ 2] , \,\,\, [0, \, 1] \,\, [a_1, \, c_2]$
 
 \bigskip 

Case 2: $\,\, [a_1, \, c_2] \,\, [c_1, \, c_2]  \quad  c_1 < a_1 <  0, \,\, -c_2 < c_1 < 0, \,\, c_2 > 1 $  

$\qquad$ subcases: $\,\, [0 , \, c_2] \,\, [ c_1 , \, c_2 ] , \,\,\, [0, \, c_2 ] \,\, [a_1 , \, c_2 ]$

\bigskip

Case 3: $\,\,   [a_1,\, 1]\,\,  [0, \, c_2]  \quad  -c_2 < a_1 < 0, \,\, c_2 > 1  $ 

$\qquad$ subcases: $ \,\,  [c_1, \, 1] \,\,  [0,\, c_2] , \,\,\, [0, \, 1] \,\, [0,\, c_2], \,\,\, [a_1, \, 1] \,\, [0, \, 1 ] ,
\,\,\, [c_1, \, 1] \,\,  [0, \, 1] $

\bigskip

Case 4: $\,\,   [a_1, \, 1] \,\, [1,\, c_2] \quad   -c_2 < a_1 < 0, \,\, c_2 > 1  $

$\qquad$ coverage $\leq \, .5$, use $\,\, \phi \,\, [0, \,1]$ 

$\qquad$ subcases: $ \,\,   [c_1, \, 1] \,\, [1,\, c_2] , \,\,\, [0,\, 1] \,\, [1, \, c_2] $ 

\bigskip

Case 5: $\,\, [a_1, \, 1] \,\, [c_1, \, 1]  \quad  -1 \leq a1 < 0, \,\, a_1 < c_1 < 0  $  

$\qquad$ coverage $\leq \, .5$, use $\,\, \phi \,\,  [0, \,1] $

\bigskip

Case 6: $\quad  \phi \,\, [a_1, \, 1]  \quad -1 \leq a1 < 0 $ 

$\qquad$coverage $\leq \, .5$, use $\,\, \phi \,\, [0, \,1]$ 

$\qquad$ subcase: $ \,\, \phi \,\,  [c_1, \, 1]  $

\bigskip

Case 7:$\quad \phi \,\, [c_1, \, c_2] \quad  -c_2 < c_1 < 0, \,\,  c_2 > 1  $ 

$\qquad$ subcase: $\,\,  \phi \,\, [a_1, \, c_2] $ 

\bigskip

Case 8: $\quad  \phi \,\, [0, \, c_2]   \quad c_2 >    \quad$ coverage $\leq \, .5$, use $\,\, \phi \,\, [0, \,1] $ 

$\qquad$ subcases: $ \,\, \phi \,\,  [1, \, c_2] $ 

\bigskip\bigskip

Thus, only cases 1, 2, 3, and 7 need to be treated. Numerical optimization (discussed in Section 2)
indicates that Case 2 is always at least as good as any other. In fact,  as noted in Portnoy (2018),
it appears that for $h$ larger than a cutoff slightly less than 5 (coverage probability
about .8), the mixing probability is 1, and the non-randomized invariant interval $\, [c_1, \, c_2] \, $ is
optimal, at least according to the numerical results. 

However, as indicated in Section 3, the numeric
optimization is surprisingly difficult, and can not prove that a given rule is optimal, or even that
there is an optimal rule (as the minimax coverage may be a limit as the endpoints or probabilities
tend to their boundaries). Therefore, it remains to prove that there is an optimal invariant mixture.

By the above proof of this Corollary, we need only show that each of the two-point mixtures in 
Cases 1, 2, 3, and 7 achieve the maximum of the minimum coverage (over $\lambda$) at finite 
values for the endpoints and probability. Consider the apparent optimal rule given by Case 2. 
Since the minimizing $\lambda^*$ is finite ($\lambda^* = 0 $ or $\lambda* \in (0, \, + \infty)$), the minimal
coverage probability is continuous in $(c_1, \, a_1, p)$. Thus, the maximum will be attained as long as 
$c_1$ is bounded away from $\, - \infty $ (since $\, c_1 < a_1 < 0 $, and $\, p \in [0, \, 1]$). Since the length
is fixed, if $c_1$ were unbounded, then $p$ would need to tend to zero (along some sequence).
So consider the derivative of the coverage probability for Case 2 as $\, p \rightarrow 0 \,$. 
Using \eqref{Plampos}, the coverage probability becomes
\begin{eqnarray*} \label{Case2p}
&  & p \,  P(\lambda , \, c_1 , \, c_2) +  (1-p) \, P(\lambda, \, a_1 , \, c_2)  \, = \\
&   & \,\,\, \Phi \left( \lambda \left(1 + \frac{1}{c_2} \right) \right) + 1  - 
p \,  \Phi \left( \lambda \left(1 + \frac{1}{c_1} \right) \right) - (1-p) \, \Phi \left( \lambda \left(1 + \frac{1}{a_1} \right) \right) \, .
\end{eqnarray*}
Now $\, h = p(c_2 - c_1) + (1-p)(c_2 - a_1) \,$, or $c_2 = h + a_1 + p(c_1 - a_1) \,$. Inserting $c_2$
in the first term in \eqref{Case2p} and differentiating with respect to $p$ gives:
$$
\frac{\lambda^*(c_1 - a_1)}{(h + a_1 + p(c_1 - a_1))^2} \, \varphi \left( \lambda \left(1 + \frac{1}{c_2} \right) \right) \,
 - \,  \Phi \left( \lambda \left(1 + \frac{1}{c_1} \right) \right) + \Phi \left( \lambda \left(1 + \frac{1}{a_1} \right) \right) \, .
$$
The first term is clearly positive, and the difference in the last two is positive since $\, c_1 < a_1 \,$. 
Thus, the minimal coverage probability can {\it not} be maximized as $\, p \rightarrow 0 \,$. Hence, 
the maximum is attained at finite values. Entirely similar proofs work for the other 3 cases. 
As a consequence, from Theorem 1 and the Corollary, the optimal invariant rule is optimal among all rules.
\end{proof}

 \bigskip
 
 \section{Some details of the numerical optimization}
 
 Numerical optimization for each of the Cases above appears to be remarkably difficult and complicated. One
 minor complication is that a separate minimization over $\lambda$ is needed before the minimal coverage 
 can be maximized over the endpoints and probability variables. Fortunately, the R-function {\tt optimize} 
 (see R Core Team (2015)) appears to work quickly and efficiently for the $\lambda$-minimization, 
 especially since it is possible
 to compute an upper bound on $\lambda$
 above which the $\lambda$-derivative is positive (and so which bounds the minimizing value). 
 
 Now consider  numerically maximizing the minimal probability over the interval variables, 
 say $(c_1, \, a_1, \, c_2, \, p)$ for Case 2, subject to fixing 
 $$
  h = p(c_2 - c_1)+ (1-p)(c_2 - a_1) \,\, .
 $$
 This presents a more serious problem: the coverage probability is not differentiable when any 
 endpoint is zero. This suggests that trying to solve the equation of partial derivatives may by very 
 problematic, thus precluding the use of Lagrange multipliers to handle the length constraint. 
 As an alternative, solve the length equation for $c_2$ and use the R-function {\tt optim } (R Core Team (2015))
 to maximize  over $(c_1, \, a_1, \, p)$. Unfortunately, incorporating the constraints 
 ($ -c_2 < c_1 < a_1 < 0 \, , \,\, c_2 > 1)$ still posed numerical complications. This algorithm often worked, 
 but for some $h$-values the routine 
 indicated a failure to converge numerically, and in other cases gave very unreliable results depending on
 starting values used. Thus, an initial grid search was used (with mesh .1 in each variable), and the routine
 {\tt optim} was used on the maximizing grid rectangle. Even then, some special programming was needed to
 deal with the constraints (especially for values of $h$ less than 2.5). Nonetheless, after considerable refinement,
 the computer results appeared to be reliable, with accuracy of at least 4 decimal places. With the obvious 
 modifications, the same code was used to treat the other cases. The output provided the plot in
 Portnoy (2018), though (of course) none of the numerical results can be guaranteed. 
 
 \bigskip\bigskip
 
 For completeness, the following gives the R-code used for Case 2 (omitting modification for smaller $h$-values):
\bigskip  {\tt \\
\# case 2  [c1,c2] [c11,c2]  -c2 < c1 < 0 ; -c1 < c11 < 0  \\
\#  b[1]=c1 , b[2] = c11 \\
\# b[3] = p  ;      h = p*(c2-c1) + (1-p)*(c2-c11) \\
c2h <- function(b,h) \{ p <- b[3] \\
return( h + p*b[1] + (1-p)*b[2] ) \} \\
concheck <- function(c1,c11,c2) \{ \# check constraints \\
return( (c2 > 1 \& c1 > -c2 \& c11 > c1) ) \} \\
c1s <- -.00001 - .1*(0:200); c1s[11] <- -1 \\
c11s <- -.00001 - .1*(0:200) \\
ps <- .1*(0:10) ; ps[11] <- .99999 \\
\\
P0 <- function(lam,b,h) \{ \# b[1]= c1 , b[2] = c11 , b[3] = p \\
\# h = p*(c2-b[1]) + (1-p)*(c2-b[2]) \\
c2 <- c2h(b,h) ; p <- b[3] \\
return( p*P(lam,b[1],c2) + (1-p)*P(lam,b[2],c2) ) \} \\
\\
dP0 <- function(lam,b,h) \{ \# b[1]= c1 , b[2] = c11 , b[3] = p \\
p <- b[3] \\
\# h = p*(c2-b[1]) + (1-p)*(c2-b[2])\\
c2 <- c2h(b,h)  \\
return( p*dP(lam,b[1],c2) + (1-p)*dP(lam,b[2],c2) ) \} \\
\\
P1 <- function(b1,h=h,ret=ret) \{ \\
\# b1[1]=c1, b1[2]=c11, b1[3]=p ;  min over lambda \\
\# lam0 = new lam* on return \\
 c2 <- c2h(b1,h) ; b <- c(b1,c2) \\
 if(!(c2 > 1)) \\
   \{if(ret) return(1+runif(1)) \\
   else return(list(objective=1+runif(1),min=-1)) \} \\
\# get upper bound lam0; from earlier runs lam0 = 3 should work, \\
\# but check for P' > 0 \\
d <- -1 ; lam0 = 2 \\
while(d <= 0) \{ \\
 d <- dP0(lam0,b,h) \\
 lam0 <- lam0 + 1 \} \\
\# min over lam in [0, lam0] \\
m <- optimize(P0,c(0,lam0),b,h) \\
if(ret) return(m\$objective) else return(m) \} \\
\\
for(i in 1:length(hs)) \{ h <- hs[i] \\
\# max over grid \\
c1m <- 0 ; c11m <- 0 ; pm <- 0  ; Pm <- 0 \\
for(c1 in c1s) \{ for(c11 in c11s) \{ for(p in ps) \{ \\
b <- c(c1,c11,p) \\
c2 <- c2h(b,h)  \\
if( concheck(c1,c11,c2) ) \{  \\
  Pn <- P1(b,h,T) \# ;  print(c(Pn,c1,c11,c2,p)) \\
  if(Pn > Pm) \{ \\
    Pm <- Pn ; c1m <- c1 ; c11m <- c11 ; pm <- p \}  \}  \}\}\} \\
\# max over cell \\
c2 <- c2h(c(c1m,c11m,pm),h) \\
low <- c(max(c1m-.1,-c2),max(c11m-.1,c1),max(pm-.1,0)) \\
up <- c(min(c1m+.1,-.000001),min(c11m+.1,-.000001), \\
min(pm+.1,.999999)) \\
b0 <- (low+up)/2 \\
m <- optim(b0,P1,lower=low,upper=up, \\
control=list(fnscale=-1),method=''L-BFGS-B'',h=h,ret=T) \\
 if(m\$converge != 0) print(paste( \\
 "possible non-convergence", m\$converge =,m\$converge)) \\
 lam <- P1(m\$par,h,F) \\
 if(abs(lam\$objective - m\$value) > .000001)  \\
  print(paste( "lam min problem:" \\
  "(P*,P(lam*)):",c(m\$value,lam\$objective)) ) \\
 out1[i,] <- c(m\$value,c2h(m\$par,h),m\$par,h,lam\$min)  \} \\
\} \\
 }
\bigskip

\section{Multivariate Confidence Sets}

\smallskip

\begin{theorem} 
Let $\, X \sim {\cal{N}}_p ( \mu , \, \Sigma ) \,$. Then to achieve
$$
\inf_{\mu , \, \Sigma} P \left\{ || \mu || \leq c \, || X || \right\} \geq 1 - \alpha
$$
it suffices to take $\, c = 3.85 \, \alpha^{-1/p} \,$.
\end{theorem}

\begin{proof}

First (without loss of generality) assume $\Sigma$ is non-singular (otherwise, restate the
problem in a smaller dimensional space).

Now, let $\, \Sigma = \Gamma ' D(\gamma) \Gamma \,$ with $\Gamma$ orthonormal, and let 
$\, \gamma_0 =  \min \{ \gamma_i \} $. Define
$\, \lambda = \gamma / \gamma_0 \,$ and $\, \nu = \mu / \gamma_0 \,$.  Then
(dividing through by  $\gamma_0$), the coverage probability is
\begin{equation} \label{CP1}
CP = P \left\{ || \nu ||^2 \leq c^2 \, \sum \lambda_i \, Y_i^2 \right\}
\end{equation}
where $Y_i$ are independent ${\cal{N}}(\nu_i , \, 1)$. Hence, from the well-known representation
of a non-central Chi-square, $\, Y_i ^2 \, \sim \, \chi^2_{p+2K_i} \,$ where
$\{ K_i \}$ are independent Poissons with mean $\,  \nu_i^2 /2 \,$. Note that $\, \lambda_i \geq 1 \,$.
Then, from Oman and Zacks (1981), 
\begin{equation} \label{chi2}
 \sum \lambda_i \, Y_i^2 \, \sim \, \chi^2_{p+2K+2L}
 \end{equation}
 where $\, K \,$ is Poisson with mean $\, \delta \equiv  || \nu ||^2 / 2 \,$ and $\, L \,$ is an (independent) 
 sum of negative binomial random variables (with parameters depending on $p$ and $\lambda$).
 It follows that
 \begin{eqnarray} \label{CP2}
 1 - CP  & \leq  & P \left\{ || \nu ||^2 \geq  c^2 \, \chi^2_{p+2K} \right\} \\
& = & \sum_{k=0}^{\infty} \, \int_0^{2 \delta/c^2} \, \frac{ x^{p/2+k-1} \, e^{-x/2} } { \Gamma(p/2+k) \, 2^{p/2+k)} }
 \, \frac{\delta^k \, e^{- \delta}} { k! } \,\,  {\mbox{\text{\rm d}}}x   \nonumber \\
 & \leq & \sum_{k=0}^{\infty} \, \frac{ (2 \delta / c^2)^{p/2 + k} } { \Gamma(p/2+k) \, 2^{p/2+k)} } \, 
\frac{\delta^k \, e^{- \delta}} { k! } 
\end{eqnarray}
where the last inequality uses $\, e^{-x/2} \leq 1\,$. Now to continue, use the fact that
$\, \delta^{p/2 + 2k} \, e^{-\delta} \,$ is maximized at $\, \delta = p/2+k \,$, and use Stirling's approximation
(which is larger than the approximated $\Gamma$-function). Then
\begin{eqnarray} \label{CP3}
1 - CP & \leq & (1/c^2)^{p/2} \, \sum_{k=0}^{\infty} \, (1/c^2)^k \, \frac{ (p/2 + 2k)^{p/2+2k} }
{ 2 \pi (p/2 + k)^{p/2+k+1/2} \, e^{- (p/2+k)} \,\, k^{k+1/2} \, e^{-k} } \nonumber  \\
& = & \frac{1}{2\pi} \, \left( \frac{e}{c^2} \right)^{p/2} \, \sum_{k=0}^{\infty} \, \left( \frac{e^2}{c^2} \right)^k \,
\left( 1 + \frac{k}{p/2 + k} \right)^{p/2+k} \, \left( 1 + \frac{p}{4k} \right)^k  \nonumber \\
&\,\,& \qquadÊ\qquad \qquad  \times \,\,  (\max\{k , \, 1 \} (p/2 + \max\{k , \, 1 \}))^{-1/2}  \nonumber \\ \nonumber \\
& \leq &  \frac{1}{2\pi} \,  \left( \frac{2 e^{3/2}}{c^2} \right)^{p/2} \, \sum_{k=0}^{\infty} \, 
\left( \frac{2 e^2}{c^2} \right)^k \, \frac{1}{\max\{k , \, 1 \}}
 \end{eqnarray}
 where $\, ( 1 + k/(p/2 + k))^{p/2+k} \,$ is bounded by $2^{p/2+k}$ , and $\, (1 + p/(4k))^k \,$ is bounded by $\, e^{p/4} \,$.
 Note that replacing $k$ by $\,  \max \{ k , \, 1 \} \,$ follows from evaluating the summand at $\, k = 0 \,$.
 
 \medskip
 
 To find an explicit bound for $c$, use the expression
 $\, \sum_{k=0}^{\infty} \, u^k / \max \{ k , \, 1 \} \, = 1 - \log(1-u) \,$.
 Convergence of the sum in \eqref{CP3} requires $\, c^2 > 2 e^2 \,$. So set 
 \begin{equation} \label{c2def}
 c^2 = 2 e^2 \, \alpha^{-2/p} \, a 
 \end{equation} 
 (with $ a > 1$),
 and bound $\alpha$ by 1 when this is substituted in the log-term. Then \eqref{CP3} becomes
\begin{equation}
1 - CP \leq  \alpha \, \left( \frac{1}{2 \pi} \, \left( \frac{1}{a \sqrt{e} } \right)^{p/2} \, (1 - \log( 1 - 1/a) ) \right) \, .
\end{equation}

Thus, setting $\, a = 1 / (1 - \exp ( - 2 \pi \, e^{p/4} + 1 ) ) \,$, some algebra yields the inequality
\begin{equation}
1 - CP \leq \alpha \, a^{-p/2}  \leq \alpha \, . 
\end{equation}

Finally, since $\, p \geq 1 \,$, numerical evaluation gives $\, a \leq 1.00086 \,$ and one can choose
$\, c = 3.85 \, \alpha^{-1/p} \,$ from \eqref{c2def} to get uniform coverage $\, 1 - \alpha \,$. 
\end{proof}

\end{document}